\documentclass[12pt,twoside]{article}

\usepackage{graphicx,graphics, amsthm, amsfonts, amsbsy,amssymb, amsmath, cite, array,arydshln, hyperref, float,tikz}
\usepackage[utf8]{inputenc}

\let\OLDthebibliography\thebibliography
\renewcommand\thebibliography[1]{
	\OLDthebibliography{#1}
	\setlength{\parskip}{0pt}
	\setlength{\itemsep}{0pt plus 0.3ex}
}

\everymath{\displaystyle}

\usepackage[margin=1in]{geometry}

\allowdisplaybreaks

\newtheorem{theorem}{Theorem}[section]

\newtheorem{lemma}[theorem]{Lemma}
\newtheorem{corollary}[theorem]{Corollary}

\newtheorem{proposition}[theorem]{Proposition}

\renewenvironment{proof}{{\bfseries Proof.}}{\qed}

\begin{document}
\title{Independent domination polynomial of comaximal graphs of commutative rings}
\author{ 
	Bilal Ahmad Rather\\ 
	{\small \it School of Mathematics and Statistics, Shandong University of Technology, Zibo 255049, China}\\
	\texttt{bilalahmadrr@gmail.com}
           }
\date{}

\pagestyle{myheadings} \markboth{Bilal Ahmad Rather}{Independent domination polynomial of comaximal graphs of commutative rings}
\maketitle
\vskip 5mm
\begin{abstract}
The comaximal graph $ \Gamma(R) $ of a commutative ring $R$ is a simple graph with vertex set $ R $ and two distinct vertices $ a $ and $b $ of $ \Gamma(R) $ are adjacent if and only if $ aR+bR=R $, where $ aR $ is the ideal generated by $ a $ in $ R $. 
In this article, the independent domination polynomial $ D_{i}(\Gamma(\mathbb{Z}_{n}),x) $ of $ \Gamma(\mathbb{Z}_{n}) $ is discussed, along with its unimodal and log-concave properties for certain values of $n$. 
Some auxiliary results related to $D_{i}(\Gamma(\mathbb{Z}_{n}),x)$ are presented in terms of their zeros.
In addition, we determine the independence polynomial $ I(\Gamma(\mathbb{Z}_{n}),x ) $ of $ \Gamma(\mathbb{Z}_{n}) $ for special values of $n$ and provide a general result associated with it. The bounds for the zero of  the polynomial $ I(\Gamma(\mathbb{Z}_{n}),x ) $ are established, and their log-concave and unimodal properties are examined.
\end{abstract}



\section{Introduction}
\paragraph{}
Let  $G=G(V,E)$ be a finite simple undirected graph with vertex set $V=\{v_{1},v_{2},\ldots,v_{n}\}$ and the set of unordered pairs of vertices is its edge set $E$. The number of vertices in $ V $ is the \emph{order} $ n $ of $ G $. The number of edges in $E $ is the \emph{size} $ m $ of $ G $.  Two adjacent vertices $ u $ and $ v $ are represented by $ u\sim v $. A vertex of degree $ 0 $ is said to be \textit{isolated} vertex. The set of vertices adjacent to $v\in V$ excluding $ v $, denoted by $N_{G}(v)$, is the \textit{neighbourhood} of $v$ (sometimes known as open neighbourhood). The \textit{degree} $ d_{v_{i}} $ (or simply $ d_{i} $) of the vertex $ v_{i} $ is the number of elements in the set $N_{G}(v_{i})$. A graph $ G $ is said to be $ r $-\textit{regular} if $ d_{i}=r $, for all $ i=1,\dots,n. $ The \textit{union} $ G_{1}\cup G_{2}$ of two graphs $ G_{1}=G_{1}(V_{1},E_{1}) $ and $ G_{2}=G_{2}(V_{2}, E_{2}) $ is the graph with vertex set $ V_{1}\cup V_{2} $ and edge set $ E_{1}\cup E_{2}. $ The \textit{join}  $ G_{1}\vee G_{2} $ of two graphs $ G_{1} $ and $ G_{2} $ with disjoint vertex sets $ V_{1} $ and $ V_{2} $ is the graph $ G_{1}\cup G_{2} $ together with all the edges joining each vertex of $ V_{1} $ to every vertex of $ V_{2}. $ A \textit{clique} in a graph $G$ is a subset of vertices in which every pair is mutually adjacent. The \textit{complement} of $G$, denoted by $\overline{G}$ is a graph such that two vertices in $\overline{G}$ are adjacent if and only if they are not so in $G.$
A non empty set $ S\subseteq V $ is said to be a \emph{dominating} set if every vertex in $ V\setminus S $ is adjacent to at least one  vertex in $ S. $ The minimum cardinality among all dominating sets of $ G $ is said to be the \emph{domination number} of $ G, $ denoted by $ \gamma(G). $ The domination theory of graphs is very well developed, see textbook \cite{haynes}. An \emph{independent} set in a graph $ G $ is a set of pairwise non-adjacent vertices and the cardinality of the largest independent set is called the \emph{independence number} of $ G, $ denoted by $ \alpha(G). $ An \textit{independent dominating} set of $ G $ is a vertex subset that is both dominating and independent in $ G $. The \textit{independent domination number} $ \gamma_{i}(G) $ is the minimum size of all the independent dominating sets of $ G. $ The invariants $ \gamma, \alpha $ and $ \gamma_{i} $ of $ G $ are related by the inequality  $ \gamma(G)\leq \gamma_{i}(G) \leq \alpha(G)$ \cite{haynes}. Let $ d_{i}(G,k) $ denote the number of independent dominating sets of $ G $ with cardinality $ k $. The independent domination polynomial of $ G $ is defined as
\[ D_{i}(G,x)=\sum_{k=\gamma_{i}(G)}^{\alpha(G)}d_{i}(G,k)x^{k}. \]
A root of the equation $ D_{i}(G,x)=0 $ is known as the independent domination root of $ G. $ The independent domination polynomial $ D_{i}(G,x) $ is a generating function of the number of the independent dominating sets of certain cardinalities of $ G $. 
 The independent domination polynomials and their zeros have attracted many researchers. Alwardi Shivaswamy and Soner \cite{Alwardi} presented results related to $D_{i}(G,x),$ while Dod \cite{Dod} discussed properties of the independent domination polynomial and some interesting connections to well known counting problems. A detailed survey of the independent domination and its recent advancements can be seen in \cite{akbari,goddard}. Jahari and Alikhani \cite{jahari} investigated the independent domination polynomial of some generalized compound graphs.   Lonzaga \cite{lonzaga} found the independent domination polynomial of corona of some special graphs. The independent domination polynomials of zero divisor graphs associated to finite commutative rings were studied in \cite{gursoy} and later the results were modified and their zeros were addressed in \cite{bilalarxiv}. Further results related to graph polynomials of other algebraic structures can be seen in \cite{bilal,domzero}.

The independence polynomial of $ G $, denoted by $ I(G, x) $, is defined as
\[ I(G, x)=\sum_{k=1}^{\alpha (G)}c_{k}x^{k}, \]
where $ c_{k} $ is the number of independent sets of order $ k $ in $ G $.  The independent domination number and the independent number problems are known to be NP-complete for graphs. Independence polynomial of the corona of graphs are given in \cite{levitdiscrete}, their detailed survey and open problem, conjecture is presented in \cite{levitdiscrete,levitsymmetry}. Mandrescu \cite{mandrescu} discussed unimodality of some independence polynomials via their palindromicity. Unimodality and other properties of independence polynomials for some graphs can be seen in \cite{ alavi,  wang}. A survey of the independence polynomial of a graph can be seen in \cite{levitsurvey}.

For a commutative ring $ R $ with identity $ 1\neq 0, $ Shamra and Bhatwadekar \cite{sharma} introduced comaximal graph $ \Gamma(R) $ of $ R $ with vertex set $ R $, and distinct $ a $ and $ b $ in $ \Gamma(R) $ are adjacent if and only if $ aR+bR=R,$ where $ xR $ is the ideal generated by $ x\in R. $ An integral modulo ring of order $n$ is denoted by $\mathbb{Z}_{n}.$ Several properties of rings were related to its associated comaximal graphs, like $ R $ is finite ring if and only if chromatic number of $ \Gamma(R) $ is finite, chromatic number of $ \Gamma(R) $ equals the number of the maximal ideals and the number of units of $ R $ (see, \cite{sharma}). For spectral theories of comaximal graphs, see \cite{banerjeeMatrices, afkhami, bilalijpam, bilaljaa}. Cut vertices of comaximal graphs of a commutative Artinian ring were studied in \cite{esmaili}. For some other properties about comaximal graphs of $ R $, see \cite{sinha, samei} and the references therein. Homological invariants of edge rings of comaximal graph can be seen in \cite{bilalca}.

Motivated by the investigation of the independent domination polynomial, along with unimodality, and log-concavity properties for zero-divisor graphs of commutative rings, as well as other algebraic, spectral, and homological invariants of comaximal graphs of rings.
 We continue to investigate the intriguing properties of comaximal graphs of commutative rings and the problem of determining their independence polynomial and the independent domination polynomial, with special focus on their unimodal and log-concave properties and their zeros.

The complete graph is denoted by $K_{n}$, its complement by $ \overline{K}_{n} $, and the complete bipartite by $ K_{a,n-a} $. For additional undefined terms and symbols, readers are directed to \cite{ping, haynes}.

The rest of the paper is organized as follows: The independent domination polynomial of the comaximal graph $ \Gamma(\mathbb{Z}_{n}) $ is covered in Section \ref{section 2}. 
We demonstrate that the independent domination polynomial of $ \Gamma(\mathbb{Z}_{n}) $ is unimodal and log-concave for $ n=p^{n_{1}}q^{n_{2}}r^{n_{3}}, $ where $ p, q, r$ are primes and $ n_{1}, n_{2}, n_{3} $ are non-negative integers.
Section \ref{section 3} determines the independence polynomial of $ \Gamma(\mathbb{Z}_{n}) $ for $ n=p^{n_{1}}q^{n_{2}} $.
  The Eneström-Kakeya theorem is utilized to present the bounds for zeros of $I(\Gamma(\mathbb{Z}_{n}),x)$ for some values of $n$. The log-concave and unimodal properties of $I(\Gamma(\mathbb{Z}_{n}),x)$  are examined.
Finally, in conclusion, a few comments about future work are put forward.

\section{Independent domination polynomial of comaximal graphs of integer modulo ring} \label{section 2}
\paragraph{}

For a positive integer $ n $, let $ \tau(n) $ denote the number of positive factors of $ n $ and let $n=p_{1}^{n_{1}}p_{2}^{n_{2}}\dots p_{r}^{n_{r}} $ be its \emph{canonical decomposition}, where $n_{1},n_{2},\dots,n_{r} $ are non negative integers, $r$ is positive and $ p_{1},p_{2},\dots,p_{r} $ are primes. Then
\begin{equation}\label{divisors of n}
	\tau(n)= (n_{1}+1) (n_{2}+1)\dots(n_{r}+1)=\prod_{i=1}^{r}(n_{i}+1).
\end{equation}

The \emph{Euler's totient function} $ \phi(n) $ denotes the number of positive integers less or equal to $ n $ and relatively prime to $ n $. Then, it is known that
$ \sum\limits_{d|n}\phi(d)=n, $ where $ d|n $ represents $ d $ divides $ n. $ Furthermore, if $p$ and $q$ are prime, then $\phi(pq)=\phi(p)\phi(q)$. Also, we note that $\sum_{i=1}^{\ell}\phi(p^{i})=p^{\ell}-1.$

An integer $ d $ is a proper divisor of $ n $ if $ d $ divides $ n $  and $ d\notin \{1,n\} .$ Let $ d_{1}, d_{2},\dots,d_{t} $ be the distinct proper divisors of $ n. $ Let $G_{n}$ be a simple graph with vertex set $ \{d_{1}, d_{2},\dots,d_{t}\} $ in which two distinct vertices $ d_{i} $ and $ d_{j} $ are adjacent  if and only if they are relatively prime, that is, $ (d_{i},d_{j})=1 $, for $ 1\leq i<j\leq t $, where $ (x,y) $ is the greatest common factor of $ x $ and $ y $. If the prime power factorization of $n$ is $ n=\prod_{i=1}^{r}p_{i}^{n_{i}}, $ where $ r,n_{1}\leq n_{2}\leq \dots\leq n_{r} $ are positive integers and $ p_{1}< p_{2}< \dots< p_{r} $ are prime numbers, then from \eqref{divisors of n}, it follows that the size of graph $G_{n}$ is  $ |V( G_{n})|=\prod\limits_{i=1}^{r}(n_{i}+1)-2,$ since $1$ and $n$ are not in $V(G_{n}).$  The graph $ G_{n} $ plays an important role in understanding the structure of $ \Gamma(R) $ when $R\cong \mathbb{Z}_{n}.$

For $ 1\leq i \leq t $, let
\[ A_{d_{i}}= \{ x\in \mathbb{Z}_{n} : (x,n)=d_{i} \}, \] where $ (x,n) $ denotes the greatest common divisor of $ x $ and $ n. $
We observe that $ A_{d_{i}} \cap A_{d_{j}}=\emptyset$, when $ i\neq j $, implying that the sets $ A_{d_{1}}, A_{d_{2}}, \dots, A_{d_{t}} $ are mutually disjoint and partitions the vertex set of $ \Gamma(\mathbb{Z}_{n}) $ as
\[
V(\Gamma(\mathbb{Z}_{n}))=A_{d_{1}}\cup A_{d_{2}}\cup \dots \cup A_{d_{t}}\cup \{0\} \cup U(\mathbb{Z}_{n}),
\]
where $ U(\mathbb{Z}_{n})=\{x\in \mathbb{Z}_{n}: (x,y)=1 \}. $ We recall that there are $ \phi(n) $ elements in $U(\mathbb{Z}_{n}).$ We note that the elements of cell $ A_{d_{i}} $ are of the form $ xd_{i} $, where $ x $ and $ \frac{n}{d_{i}} $ are relatively prime. Young \cite{my} calculated the cardinality of $ A_{d_{i}} $ as $\phi\left( \frac{n}{d_{i}}\right)$, for $ 1\leq i \leq t .$ The induced subgraphs $ \Gamma(A_{d_i}) $ of $ \Gamma(\mathbb{Z}_{n}) $ are complements of certain  cliques (see \cite{banerjeeMatrices, afkhami}). In particular, we have
	\begin{enumerate}
		\item $x_{i}\in A_{d_{i}} $ is adjacent to $ x_{j}\in A_{d_{j}} $ if and only if $ (d_{i},d_{j})=1. $
		\item If $ v_{i}\in A_{d_{i}} $ is adjacent to $ v_{j}\in A_{d_{j}} $ for some $ i\neq j $, then $ v_{i} $ is adjacent to every $ v_{j}\in A_{d_{j}} $.
		\item No two members of the set $ A_{d_{i}} $ are adjacent, for each $ d_{i}. $			
	\end{enumerate}

The following lemma allows us to write $ \Gamma(\mathbb{Z}_{n}) $ as joined union of certain complete graphs and null graphs. We recall that joined union is the generalization of graph operation join, defined as in introduction.
\begin{lemma}[\cite{afkhami, banerjeeMatrices}] \label{joined union of comaximal graph} For the positive integer $ n $ and its proper divisor $ d_{i} $, the following hold.
	\begin{itemize}
		\item[\bf (i)] For each $ i=1,2,\dots,t $, $ \Gamma(A_{d_{i}}) $ is isomorphic to $ \overline{K}_{\phi\left (\frac{n}{d_{i}} \right )} $, where $ \Gamma(A_{d_{i}}) $ is induced subgraph of $ A_{d_{i}}. $
		\item[\bf (ii)] The comaximal graph of $ \mathbb{Z}_{n} $ is
		\[ \Gamma(\mathbb{Z}_{n})=K_{\phi(n)}\vee \left(K_{1}\cup G_{2} \right), \]
		where $ G_{2}= G_{n}\Big [\overline{K}_{\phi\left(\frac{n}{d_{1}}\right)},\overline{K}_{\phi\left(\frac{n}{d_{2}}\right)},\dots,\overline{K}_{\phi\left(\frac{n}{d_{t}}\right)}\Big]$ is the joined union of graphs $\overline{K}_{\phi\left(\frac{n}{d_{i}}\right)}$ for $1\leq i\leq t.$
	\end{itemize}
\end{lemma}

The following result gives the iterated formula for the  independent domination polynomial of comaximal graph of $ \mathbb{Z}_{n}. $
\begin{theorem}\label{ind dom of G(zn)}
Let $ G\cong \Gamma(\mathbb{Z}_n) $ be a comaximal graph of order $ n. $ Then the following hold.
\begin{itemize}
	\item [\bf (i)] If $ n $ is product of distinct primes, then the independent domination polynomial of $ G $ is
	\[ \phi(n)x+xg(x), \]
	where $ g(x) $ is the independent domination polynomial of $ G_{2}. $
	\item[\bf (ii)] If $ n=p_{1}^{n_{1}}p_{2}^{n_{2}}\dots p_{k}^{n_{k}}, $ then the independent domination polynomial of $ G $ is
	\[ \phi(n)x+x^{p_{1}^{n_{1}-1}p_{2}^{n_{2}-1}\dots p_{k}^{n_{k}-1}}f(x), \]
	where $ f(x) $ is the independent domination polynomial of the connected component $ G_{2}. $
\end{itemize}
\end{theorem}
\noindent\begin{proof}
	Let $ \mathbb{Z}_{n} $ be a finite commutative ring of order $ n $ and  $G\cong \Gamma(\mathbb{Z}_{n}) $ be its comaximal graph. Since the  invertible  $ \phi(n) $ elements of $ \mathbb{Z}_{n} $ form the clique and are adjacent to every other vertex of $ \Gamma(\mathbb{Z}_{n}) $. So, by Lemma \ref{joined union of comaximal graph}, the structure of $ \Gamma(\mathbb{Z}_{n}) $ is
\begin{align*}
	 \Gamma(\mathbb{Z}_{n})=K_{\phi(n)}\vee \left(K_{1}\cup G_{2} \right),
\end{align*}
where $ G_{2}$ is the graph defined above. Now, assume that $ n=p_{1}p_{2}\dots p_{k} $ be the product of distinct primes. Then $ G_{n} $ is connected (Theorem 5.1, \cite{banerjeeMatrices}), so $ G_{2} $ is a connected graph. Now, it is clear that each single vertex of $ K_{\phi(n)} $ form an independent domination set, since each vertex is of degree $ n-1. $ Also, the vertex of $ K_{1} $ dominates the vertices of $ K_{\phi(n)} $, so $ K_{1} $ and any other independent dominating sets of $ G_{2} $ dominates $ G. $ Thus, the independent domination polynomial of $ G $ is
\[ \phi(n)x+xg(x), \]
where $ g(x) $ is the independent domination polynomial of $ G_{2}. $\\
For the second case, assume that $ n=\prod_{i=1}^{k}p_{i}^{n_{i}} $, where $ p_{i} $'s are primes and $ n_{i} $'s are positive integers with at least one $ n_{i}\geq 2. $ By definition of $ G_{2} $, the set of vertices generated by the ideal $ \langle p_{1}p_{2}\dots p_{k } \rangle\setminus\{0\} $ are never relatively prime with other vertices of $ G_{2} $ and they form an isolated set of the cardinality $ \frac{n}{\prod_{i=1}^{k}p_{i}}-1 $. Thus, there are total $ p_{1}^{n_{1}-1}p_{2}^{n_{2}-1}\dots p_{k}^{n_{k}-1} $ isolated vertices in $ K_{1}\cup G_{2} $ and hence they are the part of every independent domination set other than the vertices of $ K_{\phi(n)}. $ Therefore, the independent domination polynomial of $ G $ is
\[ \phi(n)x+x^{p_{1}^{n_{1}-1}p_{2}^{n_{2}-1}\dots p_{k}^{n_{k}-1}}f(x), \]
where $ f(x) $ is the independent domination polynomial of the connected component of $ G_{2}. $
\end{proof}


Theorem \ref{ind dom of G(zn)} does not exactly calculate the independent domination polynomial of $ \Gamma(\mathbb{Z}_{n})  $ rather it gives its prototype, since $f(x) $ (or  $g(x)$) remains yet to be determined. Thus, the independent domination polynomial of $ \Gamma(\mathbb{Z}_{n})  $ can be completely determined if the structure of $ G_{2} $ is known, which in general seems a difficult task. Following are the immediate consequences of Theorem \ref{ind dom of G(zn)} and present the independent domination polynomial of some special values of $n.$ 

\begin{corollary}\label{first consequence}
Let $ \Gamma(\mathbb{Z}_{n}) $ be the comaximal graph of $ \mathbb{Z}_{n} $. Then the following hold.
\begin{enumerate}
\item[\bf (i)] If $ n=p $ is a prime, then the independent domination polynomial of $ \Gamma(\mathbb{Z}_{n}) $ is $ px $.
\item[\bf (ii)] If $ n=pq, $ where  $  (p<q) $ are primes, then the independent domination polynomial of $ \Gamma(\mathbb{Z}_{n}) $ is 
\[ D_{i}(\Gamma(\mathbb{Z}_{n}),x)=(pq-p-q+1)x+ x^{p}+x^{q}. \]
\item[\bf (iii)] If $ n=pqr, $ where  $  (p<q<r) $ are primes, then the independent domination polynomial of $ \Gamma(\mathbb{Z}_{n}) $ is 
\[ D_{i}(\Gamma(\mathbb{Z}_{n}),x)=x\phi(n)+x^{p+q+r-2}+x^{pq}+x^{pr}+x^{qr}. \]
\end{enumerate}
\end{corollary}
\noindent\begin{proof}
	 Let $ G\cong \Gamma(\mathbb{Z}_{n}) $ be the comaximal graph of order $ n. $ Then, we have the following cases.
\begin{itemize} 
\item[\bf (i)]  For $ n=p $, $ G_{2} $ is empty, since its vertex set is partitioned into  $V_{1}=U(\mathbb{Z}_{n})$ and $V_{2}= \{0\}$. Thus, $ \Gamma(\mathbb{Z}_{n})\cong K_{p-1}\vee K_{1}\cong K_{p} $. So, the independent domination polynomial of $G$ is $ px. $
\item[\bf (ii)] For $ n=pq, $ with $ p<q $ being primes. The proper divisors of $n$ are $ p $ and $ q $. We consider the following vertex partition of $G$.
\begin{align*}
V_{1}&=\{ v\in V(G): d_{v}=pq-1\}, \qquad V_{2}=\{v\in V(G): d_{v}=pq-p-q+1\},\\
V_{3}&=\{ zq : z=1,2,\dots,p-1\}, \qquad V_{4} = \{ zp : z=1,2,\dots, q-1 \}.
\end{align*}
Thus, the vertex set of $ G $ is portioned into $ V_{1}\cup V_{2}\cup V_{3}\cup V_{4} $ and we have the following cases:\\
Case (1). Each vertex of $ V_{1} $ is independent and dominates all the vertices of $ G $. So, $ d_{i}(G, 1)=\phi(n) $.\\
Case (2). Consider $ B=V_{2}\cup V_{3} $, where $ V_{2} $ is the independent set and dominates vertices of $ V_{1} $. Similarly, $ V_{3} $ is the independent set, dominating both $ V_{1} $ and $ V_{4} $. Thus, $ B $ is the independent dominating set of cardinality $ p $.\\
Case (3). As in case (2), it is clear that $ V_{2}\cup V_{4}  $ is the another independent dominating set of cardinality $ q. $ Therefore, the independent domination polynomial of $ G $ is
\[ D_{i}(G,x)=(pq-p-q+1)x+x^{p}+x^{q}. \]
\item[\bf (iii)] For $ n=pqr, $ with $ p<q<r $ being primes. The proper divisors of $ n $ are $ p,q,r, pq, pr $ and $ qr $. Based on these divisors, we have the following vertex partition of $ G $.
\begin{align*}
	V_{1}&=\{ v\in V(G): d_{v}=n-1\}, \qquad V_{2}=\{v\in V(G): d_{v}=\phi(n)\},\\
	V_{3}&=\{ zqr : z=1,2,\dots,p-1\}, \qquad V_{4} = \{ zpr : z=1,2,\dots, q-1 \},\\
	V_{5}&=\{ zpq : z=1,2,\dots,r-1\}, \qquad V_{6} = \{ zr : z=1,2,\dots, pq-1, p \nmid z, q\nmid z \},\\
	V_{7}&=\{ zq : z=1,2,\dots, pr-1, p \nmid z, r\nmid z \}, \quad V_{8}=\{ zp : z=1,2,\dots, qr-1, q \nmid z, r\nmid z \}.
\end{align*}
Depending on this vertex partition, we consider the following cases:\\
Case (1). Each single vertex of $ V_{1} $ is independent and dominates all other vertices of $ G, $ so $ d_{i}(G,1) =\phi(n)$\\
Case (2). Consider $ S_{1}=V_{2}\cup V_{3}\cup V_{4}\cup V_{5} $. Clearly all of them are independent. $ V_{2} $ dominates vertices of $ V_{1} $, $ V_{3} $ dominates vertices of $ V_{1} $ and $ V_{8} $, $ V_{4} $ dominates vertices of $ V_{1} $ and $ V_{7} $ and $ V_{5} $ dominates vertices of $ V_{1} $ and $ V_{6} $. Thus, $ S_{1} $ is the independent dominating set of cardinality 
$$ |S_{1}|=1+p-1+q-1+r-1=p+q+r-2. $$
Case (3). Let $ S_{2}=V_{2}\cup V_{3}\cup V_{4} \cup V_{6}. $ Then all $ V_{i} $'s are independent and always dominate $ V_{1} $. Also, $ V_{3} $ dominates $ V_{8}, $  $ V_{4} $ dominates $ V_{7} $ and $ V_{6} $ dominates $ V_{5}, V_{7} $ and $ V_{8}. $ Thus $ S_{2} $ is the independent dominating set of cardinality $$ |S_{2}|=1+p-1+q-1+(p-1)(q-1)=pq. $$\\
Case (4). Similar as case (3), it is easy to verify that $ S_{3}=V_{2}\cup V_{3}\cup V_{5}\cup V_{7}, $ and $ S_{4}=V_{2}\cup V_{4}\cup V_{5}\cup V_{8} $ are the other two independent dominating sets of $ G$ with cardinalities $ pr$ and $ qr$, respectively. Therefore, with the above calculation, the independent domination polynomial of $ G $ is
\[ D_{i}(G,x)=x\phi(n)+x^{p+q+r-2}+x^{pq}+x^{pr}+x^{qr}. \]
\end{itemize}
\end{proof}


We explain Corollary \ref{first consequence} with the help of the following example.\\
\textbf{Example 1.} Let $ \Gamma(\mathbb{Z}_{n}) $ be the comaximal graph of order $ n\geq 2 $. Then, we have the following cases.\\
(1). For $ n=5, $ $ \Gamma(\mathbb{Z}_{n})\cong K_{5} $ and its independent domination polynomial is $D_{i}(\Gamma(\mathbb{Z}_{n}),x)=5x. $\\
\begin{figure}[H]
	\centerline{\scalebox{.3}{\includegraphics{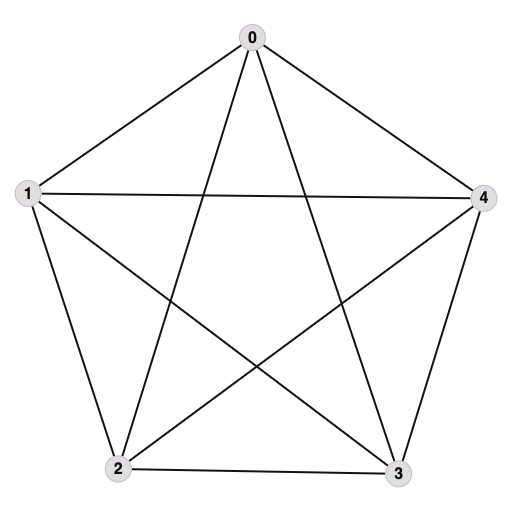}}\scalebox{.3}{\includegraphics{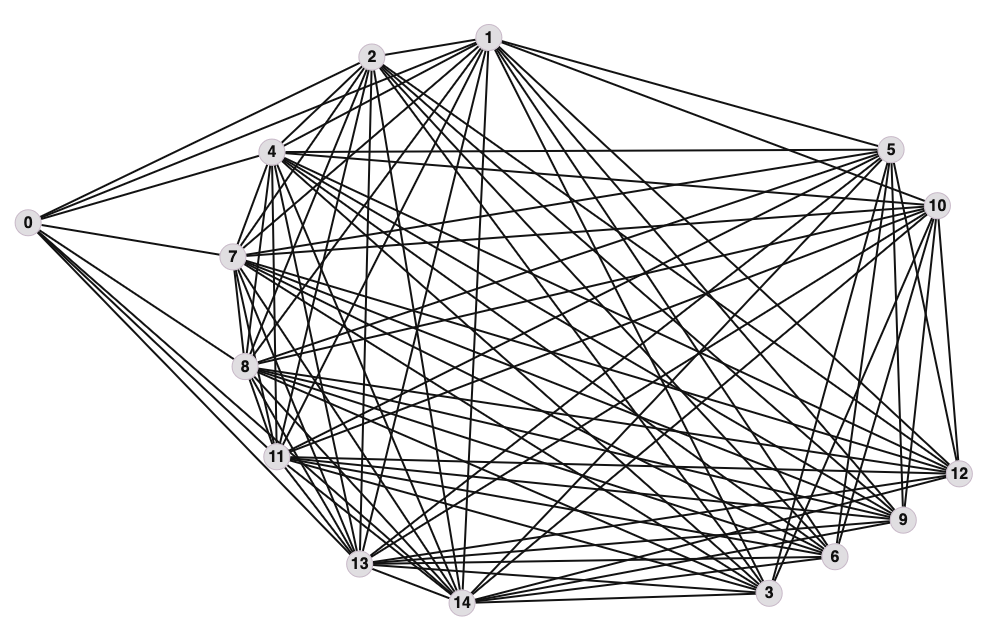}}}
	\caption{Comaximal graphs of $ \mathbb{Z}_{5} $ and $ \mathbb{Z}_{15} $}\label{comaximal graph Fig 1}
\end{figure}
\noindent(2). For $ n=15=3\cdot 5 $, it is easy to see that $ V_{1}=\{1,2,4,7,8,11,13,14\}, ~ V_{2}=\{0\},~ V_{3}=\{5,10\} $ and $ V_{4}=\{3,6,9,12\}. $ The graph structure of $ \Gamma(\mathbb{Z}_{15}) $ is given in Figure \ref{comaximal graph Fig 1}. Now, it is easy to see that the independent domination polynomial of $ \Gamma(\mathbb{Z}_{15}) $ is
\[ D_{i}(\Gamma(\mathbb{Z}_{15}), x)=8x+x^{3}+x^{5}. \]
\begin{figure}[H]
	\centerline{\scalebox{.23}{\includegraphics{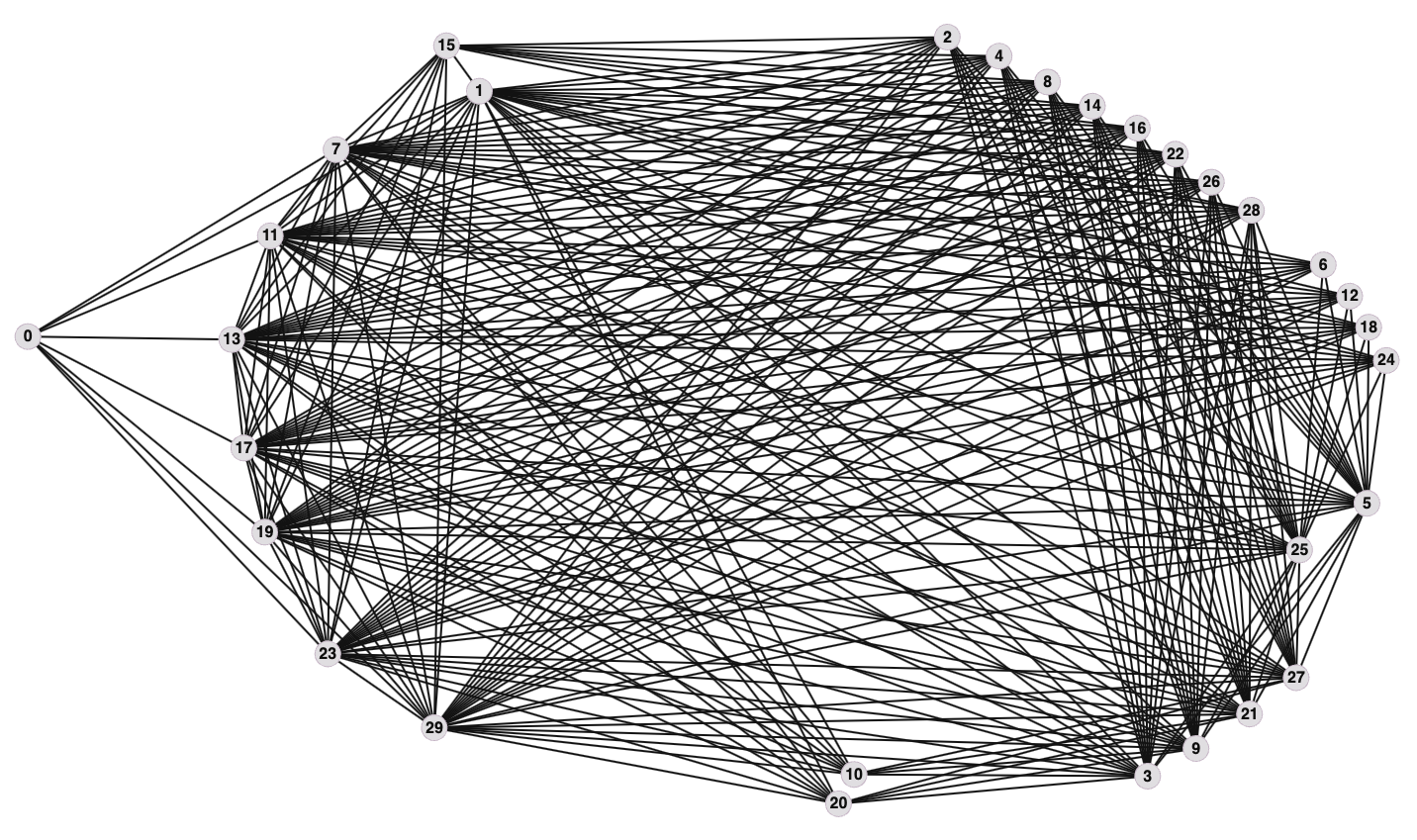}}}
	\caption{Comaximal graphs of $ \mathbb{Z}_{30} $}
	\label{comaximal graph Fig 2}
\end{figure}
\noindent(3). For  $ n=30=2\cdot3\cdot5 $, the comaximal graph $ \Gamma(\mathbb{Z}_{30}) $ is shown in Figure \ref{comaximal graph Fig 2}. By vertex partition of $ \Gamma(\mathbb{Z}_{30}) $, we get: 
 \begin{align*}
	 V_{1}&=\{1,7,11,13,17,19,23,29\},~ V_{2}=\{0\},~V_{3}=\{15\},~V_{4}=\{10,20\},~V_{5}=\{6,12,18,24\},\\
	 V_{6}&=\{5,25\},~V_{7}=\{3,9,21,27\},~ 
 V_{8}=\{2,4,8,14,16,22,26,28\}. 
\end{align*} According to the construction of the independent domination sets given in (iii) of Corollary \ref{first consequence}, the independent domination polynomial of $ \Gamma(\mathbb{Z}_{30}) $ is
\[ D_{i}(\Gamma(\mathbb{Z}_{30}),x)=8x+x^{6}+x^{8}+x^{10}+x^{15}. \]

We note that $ x=0 $ is always zero of $ D_{i}(\Gamma(\mathbb{Z}_{n}),x), $ and for $ n=pq,  $ we have the following result.
\begin{proposition}
	Let $ D_{i}(\Gamma(R),x) $ be the independent domination polynomial of ring $ R $.  If $ R\cong \mathbb{Z}_{pq}, $ $ (2<p<q ) $ are primes, then $ D_{i}(\Gamma(R),x)=0 $ has only one real root $ x=0. $
\end{proposition}
\noindent \begin{proof}
	For $ R\cong Z_{pq}, $ where $ p>2~ (p<q) $ are primes. Then, the independent domination polynomial of $ D_{i}(\Gamma(\mathbb{Z}_{pq}),x) $ is
\[ D_{i}(\Gamma(\mathbb{Z}_{pq},x)=x h(x), \]
where $h(x)=pq-p-q+1+x^{p-1}+x^{q-1}. $ Now, since $ 2<p<q, $ so $ q-p  $ is even and hence for this reason, $ h(x) $ has no real root.
\end{proof}

Next, we have more general consequences of Theorem \ref{ind dom of G(zn)}
\begin{corollary}\label{second consequence}
Let $ \Gamma(\mathbb{Z}_{n}) $ be the comaximal graph of $ \mathbb{Z}_{n} $. Then the following hold.
\begin{itemize}
\item[\bf (i)] If $ n=p^{m}, $  $  p $ is a prime and $ m\geq 2 $ is a positive integer, then the independent domination polynomial of $ \Gamma(\mathbb{Z}_{n}) $ is 
\[ D_{i}(\Gamma(\mathbb{Z}_{n}),x)=(p^{m}-p^{m-1})x+x^{p^{m-1}}. \]
\item[\bf (ii)] If $ n=p^{n_{1}}q^{n_{2}} $, where $ p<q $ are primes and $ n_{1}, n_{2}  $ are positive integers, then the independent domination polynomial $ \Gamma(\mathbb{Z}_{n}) $ is
\[ D_{i}(\Gamma(\mathbb{Z}_{n}),x)=\phi(n)x+x^{p^{n_{1}}q^{n_{2}-1}}+x^{p^{n_{1}-1}q^{n_{2}}}. \]
\item[\bf (iii)] If $ n=p^{n_{1}}q^{n_{2}}r^{n_{3}}, $ where $ p,q,r $ are primes and $ n_{1}, n_{2}, n_{3} $ are positive integers, then the independent domination polynomial $ \Gamma(\mathbb{Z}_{n}) $ is
\[ D_{i}(\Gamma(\mathbb{Z}_{n}),x)=\phi(n)x+x^{p^{n_{1}-1}q^{n_{2}-1} r^{n_{3}-1}}\Big(x^{p+q+r-2}+x^{pq}+x^{pr}+x^{qr}\Big). \]
\end{itemize}
\end{corollary}
\noindent\begin{proof}
	 Let $ G\cong \Gamma(\mathbb{Z}_{n}) $ be comaximal graph of the ring $ \mathbb{Z}_{n}. $ Now, we have the following cases to consider.\\
	   (i) For $ n=p^{m}, m\geq 2, $ the proper divisors of $ n $ are $ p^{i},  $ where $ i=1,\dots,m-1 $ and we note that $ p^{i} $ is not relatively prime to $ p^{j} $, for $ 1\leq i, j\leq m-1. $ Thus $ G_{2} $ is a disconnected graph with 
	   $$ |A_{p^{i}}|=\phi\left(\frac{p^{m}}{p^{i}}\right)=\phi(p^{m-i}),  \quad \text{for} \quad i=1,2,\dots, m-1. $$
	    So, the structure of $ G $ is
	\begin{align*}
		G&\cong K_{\phi(p^{m})}\vee \Big(K_{1}\cup \big(\overline{K}_{\phi(p^{m-1})}\cup \overline{K}_{\phi(p^{m-2})}\cup \dots \cup \overline{K}_{\phi(p^{2})}\cup \overline{K}_{\phi(p)}\big)\Big).
	\end{align*}
As each vertex of $ K_{\phi(n)} $ is of degree $ n-1 $, so its each singleton is the independent dominating set of $ G. $ Again, the vertices of 
$$ K_{1}\cup \big(\overline{K}_{\phi(p^{m-1})}\cup \overline{K}_{\phi(p^{m-2})}\cup \dots \cup \overline{K}_{\phi(p^{2})}\cup \overline{K}_{\phi(p)}\big) $$ form the independent dominating set of cardinality $ p^{m-1} $, since, $ \sum_{i=1}^{m-1}\phi(p^{i})=p^{m-1}-1. $ Therefore, the independent domination polynomial of $ G $ is
\[ D_{i}(G,x)=p^{m-1}(p-1)x+x^{p^{m-1}}. \]
(ii). For $ n=p^{n_{1}}q^{n_{2}} $, $ (p<q) $ are primes and $ n_{1},n_{2}\geq 2 $ are positive integers, the comaximal graph of $ \mathbb{Z}_{n} $ can be constructed as:
The proper divisors of $ n $ are $ p^{i}q^{j}, $ where $ 0\leq i\leq n_{1}, ~ 0\leq j\leq n_{2} $ and $ i+j\notin \{0, n_{1}+n_{2}\}. $ Based on these divisors, the vertex set of $ G_{2} $ can be partitioned as (see, \cite{banerjeeMatrices, bilalijpam}):
\begin{align*}
	V(G_{2})=&\big(A_{p}\cup A_{p^{2}}\cup \dots \cup A_{p^{n_{1}}}\big)\cup \big(A_{q}\cup A_{q^{2}}\cup \dots \cup A_{q^{n_{2}}}\big) \cup \Big( \cup_{j=1}^{n_{2}}A_{pq^{j}} \Big)\\
	&\quad\cup \Big( \cup_{j=1}^{n_{2}}A_{p^{2}q^{j}} \Big)\cup \dots \cup \Big( \cup_{j=1}^{n_{2}-1}A_{p^{n_{1}}q^{j}} \Big).
\end{align*}
By definition of $A_{d_{i}}$'s, we note that each vertex of $ A_{p^{i}} $ is adjacent to $ A_{q^{j}} $ and no vertex of $ A_{p^{i}q^{j}} $ to any other vertex of $ G_{2} $. So, $ A_{p^{i}} $'s form an independent set (say $ V_{2} $) of cardinality 
$$ \sum_{i=1}^{n_{1}}|A_{p_{i}}|=\sum_{i=1}^{n_{1}}\phi\left(p^{n_{1}-i}q^{n_{2}}\right)=p^{n_{1}-1}q^{n_{2}-1}(p-1).  $$
 Similarly, $ A_{q^{j}} $'s form another independent set $ V_{3} $ of cardinality $ p^{n_{1}-1}q^{n_{2}-1}(q-1) $ and $ A_{p^{i}q^{j}} $'s along with $\{0\}$ form an isolated set $ V_{1} $ of cardinality $ p^{n_{1}-1}q^{n_{2}-1} $ in $ G_{2}. $ Thus, it follows that 
 $$ G_{2}=\overline{K}_{p^{n_{1}-1}q^{n_{2}-1}-1}\cup K_{p^{n_{1}-1}q^{n_{2}-1}(p-1), p^{n_{1}-1}q^{n_{2}-1}(q-1)}. $$
  We recall that there are $ \phi(n) $ vertices in $ G $ of degree $ n-1 $ and each such vertex is an independent domination set of $ G. $ Further, each vertex of $ V_{1} $ is of degree $ \phi(n) $. Now, $ V_{1}\cup V_{2} $ is an independent set and the vertices of $ V_{1} $ dominates the vertices of degree $ n-1 $, while the vertices of $ V_{2} $ dominates the vertices of $ V_{3} $ and the vertices of degree $ n-1. $ It implies that $ V_{1}\cup V_{2} $ is the independent domination set of cardinality 
  $$ |V_{1}\cup V_{2}|=p^{n_{1}-1}q^{n_{2}-1}+p^{n_{1}-1}q^{n_{2}-1}(q-1)=p^{n_{1}-1}q^{n_{2}}. $$ Likewise, $ V_{1}\cup V_{3} $ is the another independent domination set of cardinality $ p^{n_{1}}q^{n_{2}-1}  $. Hence, the independent domination polynomial of $ G $ is
\[ D_{i}(G,x)=\phi(n)x+x^{p^{n_{1}}q^{n_{2}-1}}+x^{p^{n_{1}-1}q^{n_{2}}}. \]
(iii). Let $ n=p^{n_{1}}q^{n_{2}}r^{n_{3}}, $ where $ p,q,r $ are primes and $ n_{1}, n_{2}, n_{3} $ are positive integers. Then the proper divisors of $ n $ are 
$$ p^{i}, q^{j}, r^{k}, p^{i}q^{j}, p^{i}r^{k}, q^{j}r^{k}, p^{i}q^{j}r^{k},\quad  \text{where}\quad  i+j+k\notin \{0,n_{1}+n_{2}+n_{3}\}. $$
 Let $ V_{1} $ be set of vertices of $ G $ of degree $ n-1 $. Clearly, $ p^{i}q^{j}r^{k} $'s are not relatively prime to (as vertices in $G_{n}$ are not connected) $ p^{i}, q^{j}, r^{k}, p^{i}q^{j}, p^{i}r^{k}, q^{j}r^{k}, $ and they constitute an isolated set of cardinality $ |V_{2}|=p^{n_{1}-1}q^{n_{2}-1} r^{n_{3}-1} $ in $ G $. Also,  $ p^{i} $'s are adjacent to each of $q^{j},r^{k}, q^{j}r^{k},  $ in $ G_{n} $, where $ i=1,2,\dots,n_{1}, j=1,2,\dots, n_{2} $, $ k=1,2,\dots, n_{3}, $ and their cardinality of corresponding component in $ G_{2} $ is 
\begin{align*}
 |V_{8}|=\sum_{i=1}^{n_{1}}|A_{p^{i}}|&=\phi(p^{n_{1}-1}q^{n_{2}}r^{n_{3}})+\phi(p^{n_{1}-2}q^{n_{2}} r^{n_{3}})+\dots+\phi(pq^{n_{2}} r^{n_{3}})+\phi(q^{n_{2}} r^{n_{3}})\\
 &=\phi(q^{n_{2}})\phi( r^{n_{3}})(\phi(p^{n_{1}-1})+\phi(p^{n_{1}-2})+\dots+\phi(p)+1)\\
 &=(q^{n_{2}}-q^{n_{2}-1})( r^{n_{3}}- r^{n_{3}-1})p^{n_{1}-1}\\
 &=p^{n_{1}-1}q^{n_{2}-1}r^{n_{3}-1}(q-1)(r-1).
\end{align*}
Proceeding similar as above, we have:\\
(1).  $ q^{j} $'s are adjacent to each of $p^{i}, r^{k}, p^{i}r^{k} $ in $ G_{n} $, and the cardinality of the corresponding component in $ G_{2} $ is $$ |V_{7}|=p^{n_{1}-1}q^{n_{2}-1} r^{n_{3}-1}(p-1)(r-1). $$
(2). $ r^{k} $'s are adjacent to $ p^{i}, q^{j}, p^{i}q^{k} $ and the cardinality of their corresponding component in $ G_{2} $ is $$  |V_{6}|=p^{n_{1}-1}q^{n_{2}-1} r^{n_{3}-1}(p-1)(q-1).$$
(3). $ p^{i}q^{j} $ is adjacent to $ r^{k} $ in $ G_{n} $ and its corresponding component in $ G_{2} $ is of cardinality $$ |V_{5}|=\sum_{i=1}^{n_{1}}\sum_{j=1}^{n_{2}} |A_{p^{i}q^{j}}|=\sum_{i=1}^{n_{1}}\sum_{j=1}^{n_{2}}\phi\left(p^{n_{1}-i}q^{n_{2}-j} r^{n_{3}}\right)= p^{n_{1}-1}q^{n_{2}-1} r^{n_{3}-1}(r-1). $$
(4). $ p^{i}r^{k} $ is adjacent to $ q^{j} $ in $ G_{n} $ and the cardinality of their component in $ G_{2} $ is 
$$ |V_{4}|=p^{n_{1}-1}q^{n_{2}-1} r^{n_{3}-1}(q-1). $$
(5). $ q^{j}r^{k} $ is adjacent to $ p^{i} $ in $ G_{n} $ and the cardinality of its component in $ G_{2} $ is 
$$ |V_{3}|=p^{n_{1}-1}q^{n_{2}-1} r^{n_{3}-1}(p-1). $$
 Now, it is clear that each singleton of $ V_{1} $ is independent and dominates all vertices of $ G. $ Also, $ V_{2}\cup V_{3} \cup V_{4} \cup V_{5} $ is the independent domination set, since they always dominate vertices of $ V_{1} $, besides that $ V_{3} $ dominates $ V_{8} $, $ V_{4} $ dominates $ V_{7} $, $ V_{5} $ dominates $ V_{6} $. Its cardinality is $$ p^{n_{1}}q^{n_{2}}r^{n_{3}}\big(1+p-1+q-1+r-1\big)=p^{n_{1}}q^{n_{2}}r^{n_{3}}\big(p+q+r-2\big).$$
 
 Similarly, $ V_{2}\cup V_{3} \cup V_{4} \cup V_{6} $, $ V_{2}\cup V_{3} \cup V_{5} \cup V_{7} $ and $ V_{2}\cup V_{4} \cup V_{5} \cup V_{8} $ are the other independent dominating set of $ G. $ Therefore, the independent domination polynomial of $ G $ is
	\begin{align*}
D_{i}(G,x)&=|V_{1}|x+\Big( x^{|V_{2}|+|V_{3}|+|V_{4}|+|V_{5}|}+x^{|V_{2}|+|V_{3}|+|V_{4}|+|V_{6}|}+x^{|V_{2}|+|V_{3}|+|V_{5}|+|V_{7}|}\\
&\qquad\quad~ +x^{|V_{2}|+|V_{4}|+|V_{5}|+|V_{8}|} \Big)\\
&=\phi(n)x+x^{p^{n_{1}-1}q^{n_{2}-1} r^{n_{3}-1}}\Big(x^{p+q+r-2}+x^{pq}+x^{pr}+x^{qr}\Big).
\end{align*}
\end{proof}

It is easy to verify that for $ m=1, 2$, the independent domination polynomial $ D_{i}(\Gamma(\mathbb{Z}_{p^{m}}),x) $ of $ \Gamma(\mathbb{Z}_{p^{m}}) $  has only real zeros.  Though for $ m\geq 3, p\geq 2 $, the independent domination polynomial is $ D_{i}(\Gamma(\mathbb{Z}_{p^{m}}),x)=xh(x),$ where $ h(x)=p^{m}-p^{m-1}+x^{p^{m-2}}.  $ We state the following result about the zeros of $ h(x). $
\begin{proposition}
The polynomial $ h(x) $ has one real zero.
\end{proposition}
\noindent \begin{proof}
	Since, $ h(x)=x^{p^{m-2}}+p^{m}-p^{m-1} $ and $p^{m-2} $ is odd for prime $ p\geq 3 $ and $ m\geq 3. $ Thus $ h(x) $ has only one real zero.
\end{proof}

\begin{theorem}
	For $ m\geq 3, p\geq 2 $, the independent domination polynomial of $\Gamma(\mathbb{Z}_{n})$ with $n=p^{m}$ has only one real zero.
\end{theorem}

We illustrate the Corollary \ref{second consequence} with the help of the following example.\\
\noindent\textbf{Example 2.} Let $ \Gamma(\mathbb{Z}_{n}) $ be the comaximal graph of order $ n\geq 2 $. Then, we have the following cases.\\
(1). Let $ n=2^4=16$. Then the proper divisor of $ n $ are $ 2,4 $ and $ 8. $ Recall that the elements of $ A_{d_{i}} $ are of the form $ k\cdot d_{i}, $ where $ \left(k,\frac{n}{d_{i}}\right)=1,  $ then the $ A_{d_{i}} $ sets are
\begin{align*}
A_{2}=\{2,6,10,14\}, ~ A_{4}=\{4,12\}~ \text{and} ~ A_{8}=\{8\}.
\end{align*}
Also, vertices of $ A_{d_{i}} $'s  are not connected to each other, since the proper divisors $ 2,4 $ and $ 8 $ are not pairwise relatively prime. Thus, $ V_{2}=\{0\}\cup A_{2}\cup A_{4}\cup A_{8} $ form the independent set of cardinality $ 8 $ with each vertex having degree $ 8 $ in $ G\cong \Gamma(\mathbb{Z}_{16}) $. Besides, there are $ \phi(16)=8 $ vertices of degree $ 16 $, which constitute the vertex set  $ V_{1} $. Thus, each single vertex of $ V_{1} $ is the independent dominating vertex of $ G $. Also, $ V_{2} $ is the another independent dominating set in $ G $, see Figure \ref{comaximal graph Fig 3}. Hence, the independent domination polynomial of $ G $ is
\[ D_{i}(G,x)=8x+x^{8}. \]
\begin{figure}[H]
	\centerline{\scalebox{.28}{\includegraphics{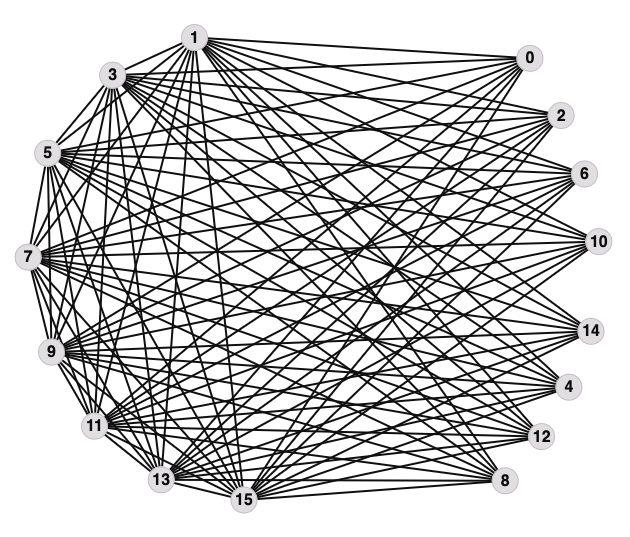}}\scalebox{.28}{\includegraphics{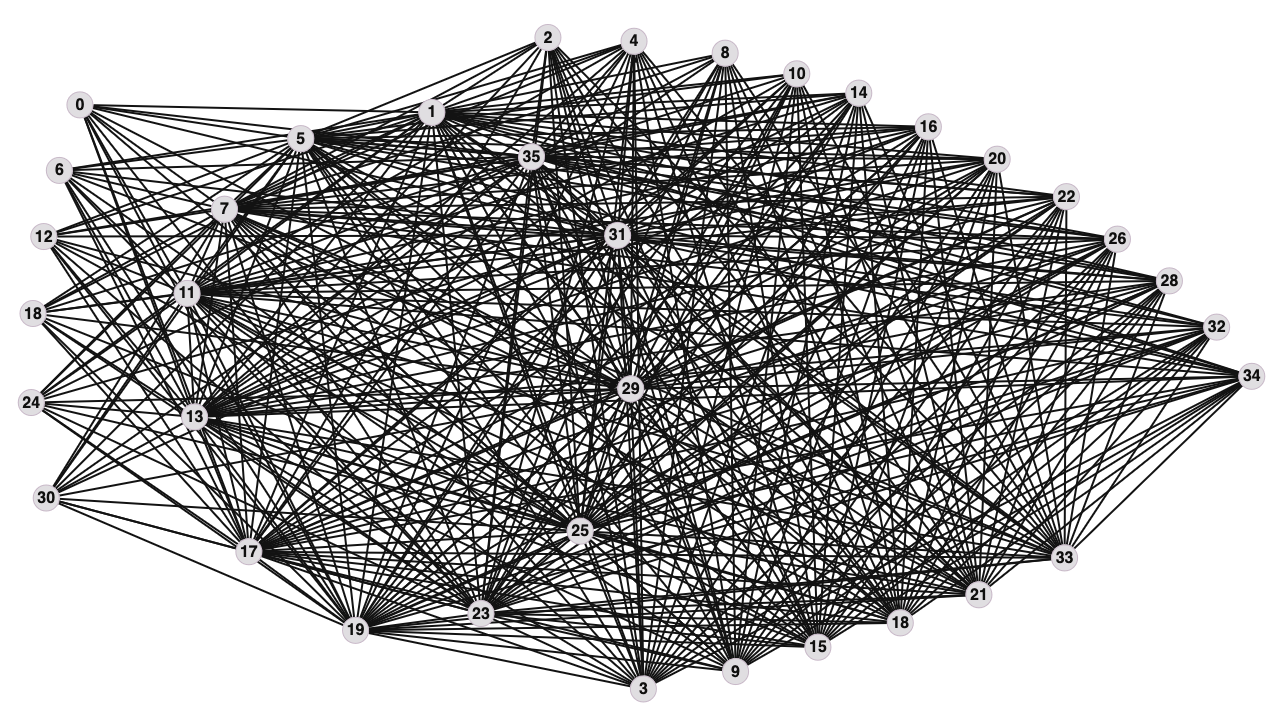}}}
	\caption{Comaximal graphs of $ \mathbb{Z}_{16} $ and $ \mathbb{Z}_{36} $}
	\label{comaximal graph Fig 3}
\end{figure}
\noindent(2). For $ n=36=2^{2}\cdot 3^{2} $, the proper divisors of $ n $ are $ 2, 2^{2}, 3, 3^{2}, 2\cdot 3, 2\cdot 3^{2} $ and $ 2^{2}\cdot 3. $ The corresponding $ A_{d_{i}} $ sets are:
\begin{align*}
A_{2}&=\{2,10,14,22,26,34\},~ A_{4}=\{4,8,16,20,28,32\},~ A_{3}=\{3,15,21,33\},\\
 A_{9}&=\{9,18\},~ A_{6}=\{6,30\},~ A_{12}=\{12,24\},~ A_{18}=\{18\}.
\end{align*}
By definition of $ G_{2} $, $ A_{6}\cup A_{12}\cup A_{18} $ form the isolated set, while $ A_{2}\cup A_{4} $ and $ A_{3}\cup A_{9} $ form the complete bipartite graph $ K_{12,6} $. The structure of $ G\cong  \Gamma(\mathbb{Z}_{36})$ is shown in Figure \ref{comaximal graph Fig 3}. Let $ V_{1} $ be the set of vertices of degree $ 36, $ $ V_{2}=\{0\}\cup A_{6}\cup A_{12}\cup A_{18}, ~ V_{3}= A_{2}\cup A_{4} $ and $ V_{4}=A_{3}\cup A_{9} $. Therefore, the independent domination polynomial of $ G $ is 
\[  D_{i}(G,x)=12x+x^{12}+x^{18}. \]
Similarly, (iii) of Corollary \ref{second consequence} can be worked out.

A finite sequence of real numbers $S= \{s_{0},s_{1},\dots,s_{n}\} $ is said to be \emph{unimodal} if there exists a positive integer  $ t ~(0\leq t\leq n), $ called the \emph{mode} of  $ S $ such that $ s_{i} $ increases up to $ i=t $ and decreases from then on, that is, 
$$ s_{0}\leq s_{1}\leq \dots \leq s_{t} \quad \text{and} \quad s_{t}\geq s_{t+1}\geq \dots \geq s_{n}. $$ A polynomial $ p(x)=\sum_{i=0}^{n} a_{i}x^{i} $ is said to be unimodal, if its sequence of coefficients $ \{a_{i}\} $ is unimodal. Equivalently, $p(x)$ is  unimodal if its coefficient sequence has a single peak, meaning that it rises to a certain point before falling.
 We note that the cases $ t= 0 $ and  $t = n $ allows the sequence to be decreasing and increasing, respectively. So, in such cases $p(x)$ is unimodal. A polynomial $ p(x) $ is symmetric if $ a_{i}=a_{n-i}, $ for $ 0\leq i\leq n $ and \emph{log-concave} if 
\begin{equation}\label{log con 1}
	 a_{i}^{2}\geq a_{i-1}a_{i+1},~  \text{for all}~ 1\leq i\leq n-1. 
\end{equation} Suppose that $ a_{i} $'s are non negative numbers and $ p(x) $ has only real zeros. Then the basic approach for unimodality and log-concave \cite{hardy} is the Newton's inequalities:
\begin{equation}\label{log con 2}
	 a_{i}^{2}\geq a_{i-1}a_{i+1}\left(1+\frac{1}{i}\right)\left(1+\frac{1}{n-i}\right),~ i=1,2,\dots,n-1. 
\end{equation}
Clearly, $ 3 + 8x + 9x^{2} + 50x^{3} + 7x^{4} + 3x^{5} $ is unimodal while $ 14 + 7x + 9x^{2} + 7x^{3} + 3x^{4} $ is not unimodal.

The number of oscillations $ \eta(p(x)) $ in the sequence of coefficients $ a_{i} $'s of $ p(x) $ is defined as the number of changes of directions (increasing vs decreasing) in the sequence. If $ p(x) = 171 + 151x + 180x^{2} + 13x^{3} +5x^{4} + 2x^{5} $, then $ \eta(p(x))=2, $ as change of decreasing directions are two. Thus, it immediately follows that  a polynomial $ p(x) $  is unimodal if its oscillation is at most one, that is, $ \eta(p(x))\leq 1. $ 

Next,  we presents results related to unimodal and log-concave property of the independent domination polynomial of comaximal graph of ring $\mathbb{Z}_{n}.$

\begin{theorem} Let $G\cong \Gamma(\mathbb{Z}_{n})$ be the comaximal graph of order $n$. Then the followings hold.
	\begin{itemize}
		\item[\bf (i)] If $ n=p^{m}, $  $  p $ is a prime and $ m\geq 1 $ is a positive integer, then $D_{i}(G,x)$ is unimodal if and only if $m=1 $ or $m=p=2$ and it is log-concave if and only if $p\neq3$ and $m\neq 2.$
		\item[\bf (ii)] If $ n=p^{n_{1}}q^{n_{2}}, $ with primes $p<q $ and $ n_{i}\geq 1 $ are positive integer, then $D_{i}(G,x)$ is unimodal if and only if $n_{1}=n_{2}=1, p=2 $ and $q=3.$ 
		Also, $D_{i}(G,x)$ is not log-concave if and only if $n_{1}=n_{2}=1, p=3$ and $q$ is arbitrary or $n_{1}=n_{2}=1, p-q=2$ or $n_{1}=2, n_{2}=1, p=2$ and $q=3.$
	\end{itemize}
\end{theorem}
\noindent\begin{proof}
	If $m=1$, then from \ref{second consequence}, its follows that $D_{i}(G,x)=px$ is unimodal. For $p=m=2$, we get $$D_{i}(G,x)=(p^{m}-p^{m-1})x+x^{p^{m-1}}=2x+x^{2},$$ which is clearly unimodal. However, if $p^{m-1}\geq 3$, then there is at least one term missing between $x$ and $ x^{p^{m-1}},$ so there are exactly two increasing oscillation   and unimodal property ceases. Again, if there is exactly one term missing between $x$ and $ x^{p^{m-1}},$ then $ D_{i}(G,x)$ cannot be log-concave. This happens if and only if $ p^{m-1}=3$, that is, $p=3$ and $m=2$. In this case $D_{i}(G,x)=6x+x^{3} $, there is $a_{2}=0$ such that $0\ngeq a_{1}a_{3}$. Thus, $D_{i}(G,x)$ is log-concave if and only if $p\neq 3$ and $m\neq 2$.\\
	(ii) For $n=p^{n_{1}}q^{n_{2}}, $ with primes $p<q $ and $ n_{i}\geq 1 $ are positive integer, the independent domination polynomial
	$D_{i}(G,x)=\phi(n)x+x^{p^{n_{1}}q^{n_{2}-1}}+x^{p^{n_{1}-1}q^{n_{2}}}$ is unimodal if and only if difference between the exponents of $x $ and $ x^{p^{n_{1}}q^{n_{2}-1}}$ is one and that of $x^{p^{n_{1}}q^{n_{2}-1}} $ and $x^{p^{n_{1}-1}q^{n_{2}}}$ is one, otherwise there exists zero terms between them and oscillations are more than one, thereby unimodal property fails. That is, same as saying $p^{n_{1}}q^{n_{2}-1}=2$ and $ p^{n_{1}-1}q^{n_{2}-1}\big( q-p \big)=1,$ which after simplification gives $n_{1}=n_{2}=1, p=2$ and $q$ is arbitrary, the second condition $p^{n_{1}-1}q^{n_{2}-1}(q-p)=1$ implies that $ p=2$ and $q=3$. Together it follows that $p=2, q=3$ and $n_{1}=n_{2}=1$. and here the unimodal property is valid. Conversely, for $p=2, q=3$ and $n_{1}=n_{2}=1$, the independent domination polynomial is $2x+x^{2}+x^{3}$, which is clearly unimodal. The polynomial $D_{i}(G,x)$ is not log-concave if and only if there is exactly one term missing between the first two non-zero terms or between the last two non zero terms, that is, $p^{n_{1}}q^{n_{2}-1}=3$ or $p^{n_{1}-1}q^{n_{2}-1}(q-p)=2$. After simplifying first case, we obtain $n_{1}=n_{2}=1, p=3$ , and there is a zero term between $1$ and $p^{n_{1}}q^{n_{2}-1}$ such that log-concavity ceases. Also, from the second condition, we get $n_{1}=n_{2}=1, p-q=2$ or $n_{1}=2, n_{2}=1, p=2$ and $q=3.$ For this condition there is a term such that $a_{p^{n_{1}}q^{n_{2}-1}+1}=0\ngeq a_{p^{n_{1}}q^{n_{2}-1}} a_{p^{n_{1}-1}q^{n_{2}}}$, so log-concavity fails. Conversely, with these conditions, it is easy to verify that log-concavity is not true.
\end{proof}

\section{Independence polynomial of comaximal graphs of $ \mathbb{Z}_{n} $}\label{section 3}
\paragraph{}
In this section, we discuss the independence polynomial of comaximal graphs of commutative ring $ \mathbb{Z}_{n} $ for various values of $ n. $ We recall some basic facts about the polynomial $I(G,x).$
 It is well known that the independence polynomial of $ K_{n}, n\geq 2 $ is $ 1+nx, $ the independence polynomial of $ K_{1} $ is $ 1+x $ and hence that of $ nK_{1} $ is $ (1+x)^{n}. $
We recall that the independence polynomials of graph operations $ G_{1}\cup G_{2} $ and $ G_{1}\vee G_{2} $ (see, \cite{gutman}) satisfy the following
\begin{equation}\label{ind poly of union}
	I(G_{1}\cup G_{2},x)=I(G_{1},x)I(G_{2},x),
\end{equation}
\begin{equation}\label{ind poly of join}
	I(G_{1}\vee G_{2},x)=I(G_{1},x)+I(G_{2},x)-1.
\end{equation}

The following result gives the independence polynomial of $ \mathbb{Z}_{n}, n\geq 2 $ for some values of $ n. $
\begin{theorem}\label{independent poly of zn}
Let $ \Gamma(\mathbb{Z}_{n}) $ be the comaximal graph of $ \mathbb{Z}_{n}, n\geq 2 $. Then the following hold.
\begin{itemize}
	\item[\bf (i)] The independence polynomial of $ \Gamma(\mathbb{Z}_{n}) $ for  $ n=p $ is   $I( \Gamma(\mathbb{Z}_{n}),x)=1+px, $ where $ p $ is prime.
	\item[\bf (ii)] For $ n=pq,$  where $ (p<q) $ are primes, the independence polynomial of $ \Gamma(\mathbb{Z}_{n}) $ is 
	\[ I(\Gamma(\mathbb{Z}_{n}),x)=\big(\phi(n)-1\big)x+(1+x)^{p}+(1+x)^{q}-1.\]
	\item[\bf (iii)] For $ n=p^{m}$  where $ p $ is prime and $ m $ is positive integer, the independence polynomial of $ \Gamma(\mathbb{Z}_{n}) $ is 
	\[ I(\Gamma(\mathbb{Z}_{n}),x)=\phi(n)x+(1+x)^{p^{m-1}}.\]
	\item[\bf (iv)] For $ n=p^{n_{1}}q^{n_{2}},$  where $ (p<q) $ are primes and $ n_{1}, n_{2} $ are positive integers, the independence polynomial of $ \Gamma(\mathbb{Z}_{n}) $ is 
	\[ I(\Gamma(\mathbb{Z}_{n}),x)=\phi(n)x+(1+x)^{p^{2(n_{1}-1)}q^{2(n_{2}-1)}(q-1)}+(1+x)^{p^{2(n_{1}-1)}q^{2(n_{2}-1)}(p-1)}-(1+x)^{p^{n_{1}-1}q^{n_{2}-1}}.\]
\end{itemize}
\end{theorem}
\noindent\begin{proof}
	 (i) For $ n=p, $ $ \Gamma(\mathbb{Z}_{n}) $ is isomorphic to $ K_{p} $ and the result follows.\\
(ii) For $ n=pq $, $ \Gamma(\mathbb{Z}_{n})=K_{\phi(n)}\vee (K_{1}\cup K_{p-1,q-1}) $ and $ K_{p-1,q-1}=\overline{K}_{p-1}\vee \overline{K}_{q-1}. $ By \eqref{ind poly of join} and \eqref{ind poly of union}, the independence polynomial of $ \overline{K}_{p-1}\vee \overline{K}_{q-1} $ is $I(\overline{K}_{p-1},x)+I(\overline{K}_{q-1},x)=(1+x)^{p-1}+(1+x)^{q-1}-1 $ and that of $ K_{1}\cup \Big(\overline{K}_{p-1}\vee \overline{K}_{q-1}\Big) $ is $ (1+x)^{p}+(1+x)^{q}-(1+x). $ Again using \eqref{ind poly of join}, the independence polynomial of $ \Gamma(\mathbb{Z}_{n}) $ is
\begin{align*}
	I(\Gamma(\mathbb{Z}_{n}),x)&=1+\phi(n)x+ (1+x)^{p}+(1+x)^{q}-(1+x)-1\\
	&=\big(\phi(n)-1\big)x+(1+x)^{p}+(1+x)^{q}-1.
\end{align*}
(iii) For $ n=p^{m}, m\geq 2, $ the comaximal graph is $ \Gamma(\mathbb{Z}_{n})\cong K_{\phi(n)}\vee \overline{K}_{p^{m-1}} $ and by \eqref{ind poly of join} of $ \Gamma(\mathbb{Z}_{n}) $ is
\begin{align*}
	I(\Gamma(\mathbb{Z}_{n}),x)&=I(K_{\phi(n)},x)+I(\overline{K}_{p^{m-1}},x)-1=1+\phi(n)x+(1+x)^{p^{m-1}}-1\\
	&=\phi(n)x+(1+x)^{p^{m-1}}.
\end{align*}
(iv) For $ n=p^{n_{1}}q^{n_{2}}, $ $ \Gamma(\mathbb{Z}_{n})\cong K_{\phi(n)}\vee \Big(\overline{K}_{p^{n_{1}-1}q^{n_{2}-1}}\cup\big(\overline{K}_{p^{n_{1}-1}q^{n_{2}-1}(q-1)}\vee \overline{K}_{p^{n_{1}-1}q^{n_{2}-1}(p-1)}\big) \Big)  $ and by repeated applications of \eqref{ind poly of join} and \ref{ind poly of union}, the independence polynomial of $\overline{K}_{p^{n_{1}-1}q^{n_{2}-1}(q-1)}\vee \overline{K}_{p^{n_{1}-1}q^{n_{2}-1}(p-1)}$ is 
\[ (1+x)^{p^{n_{1}-1}q^{n_{2}-1}(q-1)}+(1+x)^{p^{n_{1}-1}q^{n_{2}-1}(p-1)}-1, \]
and that of $ \overline{K}_{p^{n_{1}-1}q^{n_{2}-1}}\cup\big(\overline{K}_{p^{n_{1}-1}q^{n_{2}-1}(q-1)}\vee \overline{K}_{p^{n_{1}-1}q^{n_{2}-1}(p-1)}\big)$ is
\[ (1+x)^{p^{n_{1}-1}q^{n_{2}-1}}\Big((1+x)^{p^{n_{1}-1}q^{n_{2}-1}(q-1)}+(1+x)^{p^{n_{1}-1}q^{n_{2}-1}(p-1)}-1\Big). \]
Finally, the independence polynomial of $  \Gamma(\mathbb{Z}_{n}) $ is
	\begin{align*}
	 I(\Gamma(\mathbb{Z}_{n})\,x)=&(1+x)^{p^{n_{1}-1}q^{n_{2}-1}}\Big((1+x)^{p^{n_{1}-1}q^{n_{2}-1}(q-1)}+(1+x)^{p^{n_{1}-1}q^{n_{2}-1}(p-1)}-1\Big)-1\\
	 &+1+\phi(n)x\\
	  =&\phi(n)x+(1+x)^{p^{2(n_{1}-1)}q^{2(n_{2}-1)}(q-1)}+(1+x)^{p^{2(n_{1}-1)}q^{2(n_{2}-1)}(p-1)}\\
	  &-(1+x)^{p^{n_{1}-1}q^{n_{2}-1}}.
\end{align*} 
\end{proof}

With similar idea, Theorem \ref{independent poly of zn} can be generalized for $ n=p^{n_{1}}p^{n_{2}}r^{n_{3}}, $ where $ n_{i} $'s are non negative integers and $ p,q,r $ are primes.  It is interesting to investigate the zeros of the independence polynomials of comaximal graphs. The umimodality and log-concavity of  $ I(\Gamma(R),x) $  remains an open topic for comaximal graphs of arbitrary rings. 

Similar to Theorem \ref{ind dom of G(zn)}, we have the following result for the independence polynomial for the comaximal graph of $\mathbb{Z}_{n}.$

\begin{theorem}\label{ind of G(zn)}
	Let $ G\cong \Gamma(\mathbb{Z}_n) $ be the comaximal graph of order $ n. $ Then the following hold.
	\begin{itemize}
		\item [\bf (i)] If $ n $ is product of distinct primes, then the independence polynomial of $ G $ is
		\[ I(G,x)=\phi(n)x+(1+x)I(G_{2},x), \]
		where $ I(G_{2},x) $ is the independence polynomial of $ G_{2} $.
		\item[\bf (ii)] If $ n=\prod_{i=1}^{k}p_{i}^{n_{i}}, $ then the independent domination polynomial of $ G $ is
		\[ I(G,x)=\phi(n)x+(1+x)^{p_{1}^{n_{1}-1}p_{2}^{n_{2}-1}\dots p_{k}^{n_{k}-1}}I(G_{2},x), \]
		where $ I(G_{2},x) $ is the independence polynomial of the  $ G_{2}. $
	\end{itemize}
\end{theorem}

From, the above results of $I(G,x)$ for $G\cong \Gamma(\mathbb{Z}_{n})$, it follows that the independence polynomial of $\mathbb{Z}_{n}$ can be completely determined if we are able to find $ I(G_{2},x)$ for the graph $G_{2}$.  The structure of $G_{n}$ becomes very complex as  the value of $k$ increases  in $n=\prod_{i=1}^{k}p_{i}^{n_{i}}.$ So, it becomes quit technical and seems hard to find its explicit expression as we increase value of $k$ beyond $3.$

Now, for the zeros of the independence polynomial. We first consider $I( \Gamma(\mathbb{Z}_{p}),x)=1+px, $ which is clearly log-concave and unimodal. Also, its zeros are real. Now, for $n=pq$, the independence polynomial is
\[ I(\Gamma(\mathbb{Z}_{n}),x)=\big(\phi(n)-1\big)x+(1+x)^{p}+(1+x)^{q}-1.\]
Expending the above equation, we have
\begin{footnotesize}
	\begin{align}\label{eq 1}
	I(\Gamma(\mathbb{Z}_{n}),x)&= 1+pqx+\Bigg(\binom{p}{2}+\binom{q}{2}\Bigg)x^{2}+\Bigg(\binom{p}{3}+\binom{q}{3}\Bigg)x^{3}+\sum_{i=4}^{p}\binom{p}{i}+\sum_{i=4}^{q}\binom{q}{i}.
\end{align}
\end{footnotesize}
We note that $2pq\geq p(p-1)+q(q-1)$ is not true in general. However, for $p=2$ and $q=3$, it is true and  for $q\geq 3$, it is false. So, $I(\Gamma(\mathbb{Z}_{2\cdot 3}),x)$ is unimodal and log-concave. For $q\geq 3$, it is not unimodal. Also, it is not necessarily log-concave for $q\geq 3$ (can be seen for $q=7$). Thus, in general $I(\Gamma(\mathbb{Z}_{n}),x)$ is neither unimodal nor log-concave. It is interesting to characterize the values of $p$ and $q$ such that  $I(\Gamma(\mathbb{Z}_{n}),x)$ is neither unimodal nor log-concave.

We recall the famous Eneström-Kakeya theorem (see, \cite{barbeau}) for finding the region in plane for the zero of a polynomial. For a polynomial $p(x)=\sum_{i=0}^{n} \alpha_{i}x^{i}$ with positive coefficients,  its zeros lie in the annulus $a\leq |Z|\leq b,$
where 
$$a=\min \left \{\left|\frac{ \alpha_{i}}{ \alpha_{i+1}}\right| : 0\leq i\leq n-1  \right \} ~\text{and}~b=\max \left  \{\left|\frac{ \alpha_{i}}{ \alpha_{i+1}}\right| : 0\leq i\leq n-1  \right \}.$$
 From Equation \ref{eq 1}, it is clear that $0<a<1$ and $b=\frac{q(q-1)}{2q}=\frac{q-1}{2}.$ Therefore, the non-zero zeros of $I(\Gamma(\mathbb{Z}_{n}),x)$ given in \eqref{eq 1} lies in the region $0<|Z|<\frac{q-1}{2}.$ We illustrate it with the following example.\\
For $p=7$ and $q=11$, the independence polynomial is given by
\begin{align*}
	I(\Gamma(\mathbb{Z}_{7\cdot 11}),x)&=1 + 77 x + 76 x^2 + 200 x^3 + 365 x^4 + 483 x^5 + 469 x^6 + 331 x^7 + 
	165 x^8\\
	&\quad + 55 x^9 + 11 x^{10} + x^{11}.
\end{align*}
It is obvious that above polynomial is neither log-concave nor unimodal. The approximated zeros of the above polynomial are 
\begin{align*}
	-0.013152, &2.49122 \pm 0.488584 i, -1.98695 \pm 1.24354 i, -1.09244 \pm 1.50704 i,\\
	& -0.18907 \pm 1.25975 i, 0.266252 \pm 0.55771 i.
\end{align*}
The above zeros lie in $0<|Z|<5$. Pictorially, these zero are represented by blue dots in Figure \ref{Fig zeros}.
\begin{figure}[H]
	\centerline{\scalebox{.3}{\includegraphics{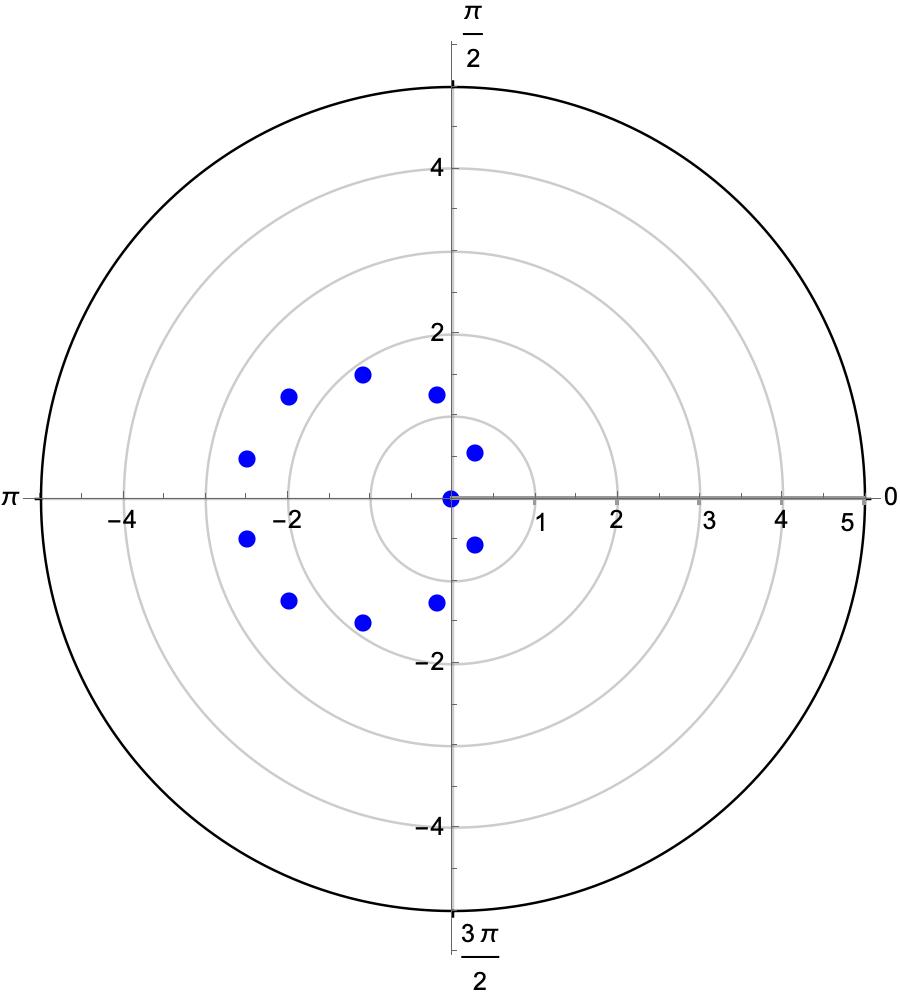}}\qquad\scalebox{.3}{\includegraphics{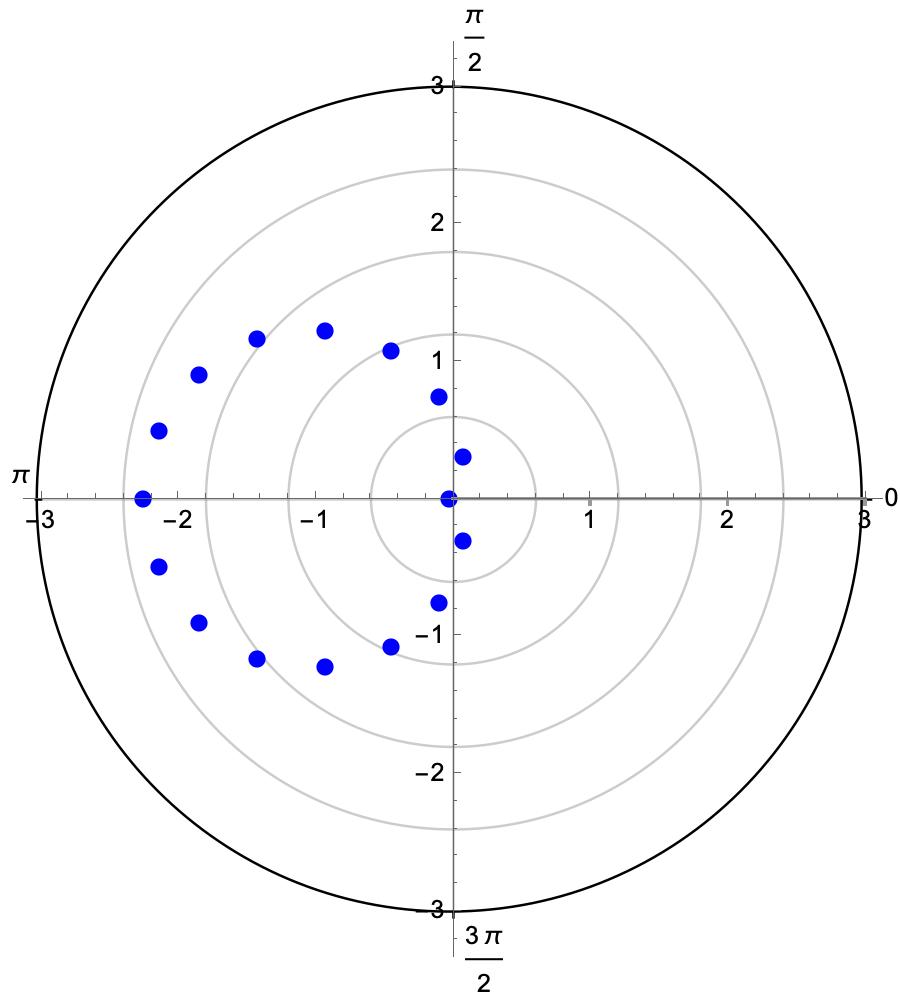}}}
	\caption{Representation of zeros of $I(\Gamma(\mathbb{Z}_{7\cdot 11}),x)$ and $I(\Gamma(\mathbb{Z}_{2^{5}}),x)$ on plane.}
	\label{Fig zeros}
\end{figure}

From Figure \ref{Fig zeros}, the zeros of $I(\Gamma(\mathbb{Z}_{7\cdot 11}),x)$ form an interesting pattern and  lie on boundary of some disc centred on negative x-axis. By computations, we observed this patter for large values of $p$ and $q.$  So, it will of interest to find the curve of the limit of zeros.

Next, for $I(\Gamma(\mathbb{Z}_{p^{m}}),x)=\phi(n)x+(1+x)^{p^{m-1}},$ we have
\begin{align*}
	I(\Gamma(\mathbb{Z}_{p^{m}}),x)=1+p^{m}x+\binom{p^{m-1}}{2}x^{2}+\binom{p^{m-1}}{3}x^{2}+\dots+x^{p^{m-1}}.
\end{align*}
It easily follows from the pattern of the binomial coefficients that the oscillation is exactly one and the polynomial is unimodal. However, it is not log-concave as 
\[  \binom{p^{m-1}}{2}^{2}\ngeq p^{m} \binom{p^{m-1}}{3}, \]
as can be verified for $p=2$. Also, by Eneström-Kakeya theorem, its zeros lie in $0<|Z|<\frac{p^{m-1}-1}{2}.$ For $2^{5}$, we have 
\begin{align*}
	I(\Gamma(\mathbb{Z}_{2^{5}}),x)&=1 + 32 x + 120 x^2 + 560 x^3 + 1820 x^4 + 4368 x^5 + 8008 x^6 + 
	11440 x^7 + 12870 x^8\\
	&\quad + 11440 x^9 + 8008 x^{10} + 4368 x^{11} + 
	1820 x^{12} + 560 x^{13} + 120 x^{14} + 16 x^{15} + x^{16},
\end{align*}
and its zeros are 
\begin{align*}
	-&2.25107, -0.0352171, -2.14742 \pm 0.494404 i, -1.85454 \pm  0.904433 i, -1.42351 \pm 1.16039 i,\\
	& -0.92986 \pm 1.2194 i, -0.460985 \pm 1.07302 i, -0.10302 \pm  0.749046 i, 0.0624826 \pm 0.307833 i.
\end{align*}
The zeros of above polynomial are shown in  Figure \ref{Fig zeros} right side and they lie in region $0<|Z|<7.5$. Also, we note that $I(\Gamma(\mathbb{Z}_{2^{5}}),x)$ is unimodal but not log-concave as $120^{2}\ngeq 32\cdot 560.$

\section{Conclusion}
The current articles present the results on the independence polynomial and the independent domination polynomial for comaximal graphs of the commutative ring $ \mathbb{Z}_{n}. $
The independent domination polynomials $ D_{i}(\Gamma(\mathbb{Z}_{n}),x) $ for $ n\in \{p^{n_{1}}, p^{n_{1}}q^{n_{2}}, p^{n_{1}}q^{n_{2}}r^{n_{3}}\}$ were presented explicitly. The polynomial $ D_{i}(\Gamma(\mathbb{Z}_{n}),x) $ is unimodal and log-concave with $ n\in \{p^{n_{1}}, p^{n_{1}}q^{n_{2}}\}$. Similarly, the independence polynomial $ I(\Gamma(\mathbb{Z}_{n}),x) $ were found for $ n\in \{p^{n_{1}}, p^{n_{1}}q^{n_{2}}\}. $
The problem of determining the independent domination polynomial/independence polynomial, as well as properties such as unimodality (log-concave), remains open for general values of $n. $ The most intriguing aspect is discovering relationships between these polynomials and their corresponding rings. The location of zeros (or their bounds) in such polynomials for general $n$ is a non-trivial remaining part.

\section*{Data Availability}
There is no data associated with this article.

\section*{Conflict of interest}
The authors declare that they have no competing interests.

\section*{Acknowledgement}
The author expresses sincere gratitude to the anonymous referees for their helpful comments and recommendations, which have improved the article's current quality.

\section*{Note} This article was accepted for publication in the journal ``Algebra Colloquium" on April 16, 2025, and is about to appear online in the second issue of 2026. Please direct any inquiries or comments at \href{bilalahmadrr@gmail.com}{bilalahmadrr@gmail.com}.

\end{document}